\documentclass{amsart}

\usepackage{amsmath}
\usepackage{amsfonts}
\usepackage{amssymb,enumerate}
\usepackage{amsthm}
\usepackage[all]{xy}
\usepackage{hyperref}

\DeclareMathOperator{\hight}{ht}

\newcommand{\Spec}{\operatorname{Spec}}

\newcommand{\td}{\operatorname{tr.deg.}}

\renewcommand{\dim}{\operatorname{dim}}

\newcommand{\fq}{\frak{q}}
\newtheorem{thm}{Theorem}[section]
\newtheorem{cor}[thm]{Corollary}

\begin{document}

\bibliographystyle{amsplain}

\date{}

\author{Parviz Sahandi}

\address{School of Mathematics, Institute for
Research in Fundamental Sciences (IPM), P.O. Box: 19395-5746,
Tehran, Iran and Department of Mathematics, University of Tabriz,
Tabriz, Iran} \email{sahandi@ipm.ir}



\thanks{This research was in part supported by a grant from IPM (No.
900130059)}


\title[quasi-Pr\"{u}fer and UM$t$ domains]{On quasi-Pr\"{u}fer and UM$t$ domains}

\begin{abstract} In this note we show that an integral domain $D$ of finite $w$-dimension is a quasi-Pr\"{u}fer
domain if and only if each overring of $D$ is a $w$-Jaffard domain.
Similar characterizations of quasi-Pr\"{u}fer domains are given by
replacing $w$-Jaffard domain by $w$-stably strong S-domain, and
$w$-strong S-domain. We also give new characterizations of UM$t$
domains.
\end{abstract}

\maketitle

\section{Introduction}

The quasi-Pr\"{u}fer notion was introduced in \cite{ACE} for rings
(not necessarily domains). As in \cite{FHP}, we say that an integral
domain $D$ is a \emph{quasi-Pr\"{u}fer domain} if for each prime
ideal $P$ of $D$, if $Q$ is a prime ideal of $D[X]$ with $Q\subseteq
P[X]$, then $Q=(Q\cap D)[X]$. It is well known that an integral
domain is a Pr\"{u}fer domain if and only if it is integrally closed
and quasi-Pr\"{u}fer \cite[Theorem 19.15]{G}. There are several
different equivalent condition for quasi-Pr\"{u}fer domains (c.f.
\cite{FHP, ACE, AJ}).

On the other hand as a $t$-analogue, an integral domain $D$ is
called a \emph{UM$t$ domain} \cite{HZ}, if every upper to zero in
$D[X]$ is a maximal $t$-ideal and has been studied by several
authors (see \cite{FGH}, \cite{DHLRZ}, and \cite{S1}). UM$t$ domains
are closely related to quasi-Pr\"{u}fer domains in the sense that a
domain $D$ is a UM$t$ domain if and only if $D_P$ is a
quasi-Pr\"{u}fer domain for each $t$-prime ideal $P$ of $D$
\cite[Theorem 1.5]{FGH}. And the other relation is the
characterization of quasi-Pr\"{u}fer domains due to Fontana, Gabelli
and Houston \cite[Corollary 3.11]{FGH}; a domain $D$ is a
quasi-Pr\"{u}fer domain if and only if each overring of $D$ is a
UM$t$-domain.

In \cite{S} we defined and studied the $w$-Jaffard domains and
proved that all strong Mori domains (domains that satisfy the ACC on
$w$-ideals) and all UM$t$ domains of finite $w$-dimension, are
$w$-Jaffard domains. In \cite{Sah} we defined and studied a subclass
of $w$-Jaffard domains, namely the $w$-stably strong S-domains and
showed how this notion permit studies of UM$t$ domains in the spirit
of earlier works on quasi-Pr\"{u}fer domains. The aim of this paper
is to prove that, for a domain $D$ with some condition on
$w$-$\dim(D)$, the following statements are equivalent, which gives
new descriptions of quasi-Pr\"{u}fer domains; a result reminiscent
of the well-known result of Ayache, Cahen and Echi \cite{ACE} (see
also \cite[Theorem 6.7.8]{FHP}).

\begin{itemize}
\item[(1)] Each overring of $D$ is a
$w$-stably strong S-domain.

\item[(2)] Each overring of $D$ is a
$w$-strong S-domain.

\item[(3)] Each overring of $D$ is a
$w$-Jaffard domain.

\item[(4)] Each overring of $D$ is a
UM$t$ domain.

\item[(5)] $D$ is a quasi-Pr\"{u}fer domain.
\end{itemize}

Throughout, the letter $D$ denotes an integral domain with quotient
field $K$ and $F(D)$ denotes the set of nonzero fractional ideals.
Let $f(D)$ be the set of all nonzero finitely generated fractional
ideals of $D$. Let $*$ be a star operation on the domain $D$. For
every $A\in F(D)$, put $A^{*_f}:=\bigcup F^*$, where the union is
taken over all $F\in f(D)$ with $F\subseteq A$. It is easy to see
that $*_f$ is a star operation on $D$. A star operation $*$ is
called of \emph{finite character} if $*_f=*$. We say that a nonzero
ideal $I$ of $D$ is a \emph{$*$-ideal} of $D$, if $I^*=I$; a
\emph{$*$-prime}, if $I$ is a prime $*$-ideal of $D$. It has become
standard to say that a star operation $*$ is \emph{stable} if
$(A\cap B)^{*}=A^*\cap B^*$ for all $A$, $B\in F(D)$. Given a star
operation $*$ on an integral domain $D$ it is possible to construct
a star operation $\widetilde{*}$ which is stable and of finite
character defined as follows: for each $A\in F(D)$,
$$
A^{\widetilde{*}}:=\{x\in K|xJ\subseteq A,\text{ for some
}J\subseteq D, J\in f(D), J^*=D\}.
$$
The $\widetilde{*}$-dimension of $D$ is defined as follows:
$$
\widetilde{*}\text{-}\dim(D)=\sup\{\hight (P) \mid P\text{ is a
}\widetilde{*}\text{-prime ideal of } D\}.
$$

The most widely studied star operations on $D$ have been the
identity $d$, and $v$, $t:=v_f$, and $w:=\widetilde{v}$
operations, where $A^{v}:=(A^{-1})^{-1}$, with
$A^{-1}:=(D:A):=\{x\in K|xA\subseteq D\}$.

Let $D$ be a domain and $T$ an overring of $D$. Let $*$ and $*'$ be
star operations on $D$ and $T$, respectively. One says that $T$ is
\emph{$(*,*')$-linked to} $D$ if $F^{*}=D\Rightarrow (FT)^{*'}=T$
for each nonzero finitely generated ideal $F$ of $D$. As in
\cite{DHLZ} we say that $T$ is $t$-linked to $D$ if $T$ is
$(t,t)$-linked to $D$. As in \cite{DHLRZ} a domain $D$ is called
\emph{$t$-linkative} if each overring of $D$ is $t$-linked to $D$.
As a matter of fact $t$-linkative domains are exactly the domains
such that the identity operation coincides with the $w$-operation,
that is \emph{DW-domains} in the terminology of \cite{Mim}.

If $F\subseteq K$ are fields, then $\td_F(K)$ stands for the
\emph{transcendence degree} of $K$ over $F$. If $P$ is a prime ideal
of the domain $D$, then we set $\mathbb{K}(P):=D_P/PD_P$.

\section{$w$-Jaffard domains}

First we recall a special case of a general construction for
semistar operations (see \cite{S}). Let $D$ be an integral domain
with quotient field $K$, let $X$, $Y$ be two indeterminates over
$D$ and $*$ be a star operation on $D$. Set $D_1:=D[X]$,
$K_1:=K(X)$ and take the following subset of $\Spec(D_1)$:
$$\Theta_1^*:=\{Q_1\in\Spec(D_1)|\text{ }Q_1\cap D=(0)\text{ or }(Q_1\cap D)^{*_f}\subsetneq D\}.$$
Set $\mathfrak{S}_1^*:=D_1[Y]\backslash(\bigcup\{Q_1[Y]
|Q_1\in\Theta_1^*\})$ and:
$$E^{\circlearrowleft_{\mathfrak{S}_1^*}}:=E[Y]_{\mathfrak{S}_1^*}\cap
K_1, \text{   for all }E\in F(D_1).$$

It is proved in \cite[Theorem 2.1]{S} that the mapping
$*[X]:=\circlearrowleft_{\mathfrak{S}_1^*}: F(D_1)\to F(D_1)$,
$E\mapsto E^{*[X]}$ is a stable star operation of finite character
on $D[X]$, i.e., $\widetilde{*[X]}=*[X]$. It is also proved that
$\widetilde{*}[X]=*_f[X]=*[X]$, $d_D[X]=d_{D[X]}$. If
$X_1,\cdots,X_r$ are indeterminates over $D$, for $r\geq2$, we let
$$
*[X_1,\cdots,X_r]:=(*[X_1,\cdots,X_{r-1}])[X_r].
$$
For an integer $r$, put $*[r]$ to denote $*[X_1,\cdots,X_r]$ and
$D[r]$ to denote $D[X_1,\cdots,X_r]$.

Let $*$ be a star operation on $D$. A valuation overring $V$ of $D$
is called a \emph{$*$-valuation overring of $D$} provided that
$F^*\subseteq FV$, for each $F\in f(D)$. Following \cite{S}, the
\emph{$*$-valuative dimension} of $D$ is defined as:
$$
*\text{-}\dim_v(D):=\sup\{\dim(V)|V\text{ is }*\text{-valuation
overring of }D\}.
$$
It is shown in \cite[Theorem 4.5]{S} that
$$
\widetilde{*}\text{-}\dim_v(D)=\sup\{w\text{-}\dim(R)|R\text{ is a
}(*,t)\text{-linked over }D\}.
$$

It is observed in \cite{S} that we have always the inequality
$\widetilde{*}$-$\dim(D)\leq\widetilde{*}$-$\dim_v(D)$. We say that
$D$ is a \emph{$*$-Jaffard domain}, if
$*\text{-}\dim(D)=*\text{-}\dim_v(D)<\infty$. When $*=d$ the
identity operation then $d$-Jaffard domain coincides with the
classical Jaffard domain (cf. \cite{ABDFK}). It is proved in
\cite{S}, that $D$ is a $\widetilde{*}$-Jaffard domain if and only
if
$$*[X_1,\cdots,X_n]\text{-}\dim(D[X_1,\cdots,X_n])=\widetilde{*}\text{-}\dim(D)+n,$$
for each positive integer $n$. In \cite{SAHA} we gave examples to
show that the two classes of $w$-Jaffard and Jaffard domains are
incomparable by constructing a $w$-Jaffard domain which is not
Jaffard and a Jaffard domain which is not $w$-Jaffard.

We are now prepared to state and prove the first main result of this
paper.

\begin{thm}\label{main} Let $D$ be an integral domain of finite $w$-dimension. Then the
following statements are equivalent:
\begin{itemize}
\item[(1)] Each overring of $D$ is a
$w$-Jaffard domain.

\item[(2)] $D$ is a quasi-Pr\"{u}fer domain.
\end{itemize}
\end{thm}

\begin{proof} $(1)\Rightarrow(2)$ Let $Q$ be a prime ideal of an overring $T$ of
$D$, and set $\fq:=Q\cap D$. Let $\tau:T_Q\to\mathbb{K}(Q)$ be the
canonical surjection and let $\iota:\mathbb{K}(\fq)\to\mathbb{K}(Q)$
be the canonical embedding. Consider the following pullback diagram:
\begin{displaymath}
\xymatrix{ D(Q):=\tau^{-1}(\mathbb{K}(\fq))=D_{\fq}+QT_Q \ar[r]
\ar[d] & \mathbb{K}(\fq) \ar[d] \\
T_Q \ar[r]^{\tau} & \mathbb{K}(Q). }
\end{displaymath}
Since $T_Q$ is quasilocal and $\mathbb{K}(\fq)$ is a DW-domain, then
$D(Q)$ is a DW-domain by \cite[Theorem 3.1(2)]{Mim}. Thus the
$w$-operation coincides with the identity operation $d$ for $D(Q)$.
Since by the hypothesis $D(Q)$ is a $w$-Jaffard domain we actually
have $D(Q)$ is a Jaffard domain. On the other hand by
\cite[Proposition 2.5(a)]{ABDFK} we have
$$
\dim_v(D(Q))=\dim_v(T_Q)+\td_{\mathbb{K}(\fq)}(\mathbb{K}(Q)).
$$
In particular $\td_{\mathbb{K}(\fq)}(\mathbb{K}(Q))$ and
$\dim_v(T_Q)$ are finite numbers. Note that by \cite[Proposition
2.1(5)]{Fon} we have $\dim(D(Q))=\dim(T_Q)$ and since
$\dim_v(D(Q))=\dim(D(Q))$, we obtain that
$$
\dim(T_Q)=\dim_v(T_Q)+\td_{\mathbb{K}(\fq)}(\mathbb{K}(Q)).
$$
Since $\dim(T_Q)\leq\dim_v(T_Q)$, then
$\td_{\mathbb{K}(\fq)}(\mathbb{K}(Q))=0$. Consequently $D$ is a
residually algebraic domain, and hence is a quasi-Pr\"{u}fer domain
by \cite[Corollary 2.8]{AJ}.

$(2)\Rightarrow(1)$ Let $T$ be an overring of $D$. We claim that $T$
is of finite $w$-dimension. Since $D$ is a quasi-Pr\"{u}fer domain,
\cite[Theorem 2.4]{DHLRZ} implies that $D$ is a $t$-linkative and
UM$t$ domain. Thus in particular $T$ is a $t$-linked overring of
$D$. Then
\begin{align*}
w\text{-}\dim(T) \leq & \sup\{w\text{-}\dim(R)|R\text{ is }t\text{-linked over }D\}\\[1ex]
   = & w\text{-}\dim_v(D)=w\text{-}\dim(D)<\infty,
\end{align*}
where the first equality is by \cite[Theorem 4.5]{S}. Finally by
\cite[Corollary 2.6]{Sah}, every UM$t$ domain of finite
$w$-dimension is a $w$-Jaffard domain to deduce that $T$ is a
$w$-Jaffard domain.
\end{proof}

As an immediate corollary we have:

\begin{cor}\label{Umt} Let $D$ be an integral domain of finite $w$-dimension. Then the
following statements are equivalent:
\begin{itemize}
\item[(1)] Each $t$-linked overring of $D$ is a
$w$-Jaffard domain.

\item[(2)] $D$ is a UM$t$ domain.
\end{itemize}
\end{cor}

\begin{proof} $(1)\Rightarrow(2)$ Let $P$ be a
$t$-prime ideal of $D$, and $T$ be an overring of $D_P$. Thus
$T=T_{D\backslash P}$ is a $t$-linked overring of $D$ by
\cite[Proposition 2.9]{DHLZ}. Therefore $T$ is a $w$-Jaffard domain
by the hypothesis. Consequently $D_P$ is a quasi-Pr\"{u}fer domain
by Theorem \ref{main}. Then $D$ is a UM$t$ domain by \cite[Theorem
1.5]{FGH}.

$(2)\Rightarrow(1)$ Let $T$ be a $t$-linked overring of $D$. Then as
the proof of Theorem \ref{main} we have
\begin{align*}
w\text{-}\dim(T) \leq & \sup\{w\text{-}\dim(R)|R\text{ is }t\text{-linked over }D\}\\[1ex]
   = & w\text{-}\dim_v(D)=w\text{-}\dim(D)<\infty.
\end{align*}
By \cite[Corollary 2.6]{Sah} we get that $T$ is a $w$-Jaffard
domain.
\end{proof}

\section{$w$-stably strong S-domains}

Let $*$ be a star operation on $D$. Following \cite{Sah} the domain
$D$ is called a \emph{$*$-strong S-domain}, if each pair of adjacent
$*$-prime ideals $P_1\subset P_2$ of $D$, extend to a pair of
adjacent $*[X]$-prime ideals $P_1[X]\subset P_2[X]$, of $D[X]$. If
for each $n\geq1$, the polynomial ring $D[n]$ is a $*[n]$-strong
S-domain, then $D$ is said to be an \emph{$*$-stably strong
S-domain}. It is observed in \cite{Sah} that a domain $D$ is
$*$-strong S-domain (resp. $*$-stably strong S-domain) if and only
if $D_P$ is strong S-domain (resp. stably strong S-domain) for each
$*$-prime ideal $P$ of $D$. Thus a strong S-domain (resp. stably
strong S-domain) $D$ is $*$-strong S-domain (resp. $*$-stably strong
S-domain) for each star operation $*$ on $D$. However, the converse
is not true in general; i.e., for some star operation $*$, the
domain $D$ might be $*$-strong S-domain (resp. $*$-stably strong
S-domain), but $D$ is not strong S-domain (resp. stably strong
S-domain). In \cite[Example 4.17]{MM} Malik and Mott gave an example
of a UM$t$ domain (in fact a Krull domain) which is not strong
S-domain. But a UM$t$ domain is a $w$-stably strong S-domain (and
hence $w$-strong S-domain as well) by \cite[Corollary 2.6]{Sah}.

We observe \cite[Corollary 2.3]{Sah} that a finite $w$-dimensional
$w$-stably strong S-domain is a $w$-Jaffard domain.

We are now prepared to state and prove the second main result of
this paper.

\begin{thm}\label{main2} Let $D$ be an integral domain of finite $w$-valuative dimension. Then the
following statements are equivalent:
\begin{itemize}
\item[(1)] Each overring of $D$ is a
$w$-stably strong S-domain.

\item[(2)] Each overring of $D$ is a
$w$-strong S-domain.

\item[(3)] Each overring of $D$ is a
UM$t$ domain.

\item[(4)] $D$ is a quasi-Pr\"{u}fer domain.
\end{itemize}
\end{thm}

\begin{proof} The implication $(1)\Rightarrow(2)$ is trivial,
and $(3)\Rightarrow(1)$ holds by \cite[Corollary 2.6]{Sah}.

$(2)\Rightarrow(4)$ Let $Q$ be a prime ideal of an overring $T$ of
$D$ and set $\fq:=Q\cap D$. As in the proof of Theorem \ref{main} we
have the following pullback diagram:
\begin{displaymath}
\xymatrix{ D(Q) \ar[r]
\ar[d] & \mathbb{K}(\fq) \ar[d] \\
T_Q \ar[r]^{\tau} & \mathbb{K}(Q). }
\end{displaymath}
Since $T_Q$ is quasilocal and $\mathbb{K}(\fq)$ is a DW-domain, then
$D(Q)$ is a DW-domain by \cite[Theorem 3.1(2)]{Mim}. Thus the
$w$-operation coincides with the identity operation $d$ for $D(Q)$.
Since by the hypothesis $D(Q)$ is a $w$-strong S-domain, we actually
have $D(Q)$ is a strong S-domain. Next we claim that $D(Q)$ is of
finite dimension. Indeed since $D(Q)$ is a DW-domain it is in fact a
$t$-linked overring of $D$. Then
\begin{align*}
\dim(D(Q))= & w\text{-}\dim(D(Q)) \\[1ex]
   \leq & \sup\{w\text{-}\dim(R)|R\text{ is }t\text{-linked over }D\}\\[1ex]
   = & w\text{-}\dim_v(D)<\infty,
\end{align*}
where the second equality is by \cite[Theorem 4.5]{S}. On the other
hand by \cite[Proposition 2.7]{ABDFK} we have the inequality belove
\begin{align*}
1+\dim(T_Q)+\min\{\td_{\mathbb{K}(\fq)}(\mathbb{K}(Q)),1\}\leq &\dim(D(Q)[X]) \\[1ex]
   = & \dim(D(Q))+1\\[1ex]
   = & \dim(T_Q)+1.
\end{align*}
The first equality holds since $D(Q)$ is strong S-domain and
\cite[Theorem 39]{K}, and the second one holds by \cite[Proposition
2.1(5)]{Fon}. Thus $\td_{\mathbb{K}(\fq)}(\mathbb{K}(Q))=0$.
Consequently $D$ is a residually algebraic domain and hence is a
quasi-Pr\"{u}fer domain by \cite[Corollary 2.8]{AJ}.

$(4)\Rightarrow(3)$ Suppose that $D$ is a quasi-Pr\"{u}fer domain
and let $T$ be an overring of $D$. Thus $T$ is also a
quasi-Pr\"{u}fer domain. Therefore $T$ is a UM$t$ domain by
\cite[Theorem 2.4]{DHLRZ}.
\end{proof}

As an immediate corollary we have:

\begin{cor}\label{Umt2} Let $D$ be an integral domain of finite $w$-valuative dimension. Then the
following statements are equivalent:
\begin{itemize}
\item[(1)] Each $t$-linked overring of $D$ is a
$w$-stably strong S-domain.

\item[(2)] Each $t$-linked overring of $D$ is a
$w$-strong S-domain.

\item[(3)] Each $t$-linked overring of $D$ is a
UM$t$ domain.

\item[(4)] $D$ is a UM$t$ domain.
\end{itemize}
\end{cor}

\begin{proof} The implication $(1)\Rightarrow(2)$ is trivial.

For $(2)\Rightarrow(4)$ let $P$ be a $t$-prime ideal of $D$, and $T$
be an overring of $D_P$. Thus $T=T_{D\backslash P}$ is a $t$-linked
overring of $D$ by \cite[Proposition 2.9]{DHLZ}. Therefore $T$ is a
$w$-strong S-domain by the hypothesis. Consequently $D_P$ is a
quasi-Pr\"{u}fer domain by Theorem \ref{main2}. Then $D$ is a UM$t$
domain by \cite[Theorem 1.5]{FGH}.

$(4)\Rightarrow(3)$ Suppose $T$ is a $t$-linked overring of $D$.
Then $T$ is a UM$t$ domain by \cite[Theorem 3.1]{S1}.

$(3)\Rightarrow(1)$ Is true by \cite[Corollary 2.6]{Sah}.
\end{proof}

Note that the equivalence $(3)\Leftrightarrow(4)$ in Theorem
\ref{main2} (resp. Corollary \ref{Umt2}) is well known
\cite[Corollary 3.11]{FGH} (resp. \cite[Theorem 2.6]{CZ}), but our
proof is completely different.

\begin{center} {\bf ACKNOWLEDGMENT}

\end{center}
I would like to thank the referee for carefully reading the first
version of this paper.


\begin{thebibliography}{10}

\bibitem{ABDFK} D. F. Anderson, A. Bouvier, D. Dobbs, M. Fontana and S. Kabbaj, {\em On Jaffard domain},
Expo. Math., {\bf 6}, (1988), 145--175.


\bibitem{ACE} A. Ayache and P. Cahen and O. Echi, {\em Anneaux
quasi-Pr\"{u}f\'{e}riens et P-anneaux}. Boll. Un. Mat. Ital {\bf
10-B}, (1996), 1--24.

\bibitem{AJ} A. Ayache and A. Jaballah, {\em Residually algebraic pairs of rings}, Math. Z. {\bf 225} (1997),
49--65.


\bibitem{CZ} G.W. Chang and M. Zafrullah, {\em The $w$-integral closure of integral domains},
J. Algebra, {\bf 295}, (2006), 195--210.

\bibitem{DHLZ} D. E. Dobbs, E. G. Houston, T. G. Lucas and M. Zafrullah, {\em T-linked
overrings and Pr\"ufer v-multiplication domains}, Comm. Algebra {\bf
17} (1989), 2835--2852.

\bibitem{DHLRZ} D. E. Dobbs, E. G. Houston, T. G. Lucas M. Roitman, and M. Zafrullah,
{\em On t-linked overrings}, Comm. Algebra {\bf 20}. No. 5, (1992),
1463--1488.



\bibitem{Fon} M. Fontana, {\em Topologically defined classes of commutative rings}, Ann. Mat. Pura Appl. {\bf
123}, (1980), 331--355.

\bibitem{FGH} M. Fontana, S. Gabelli and E. Houston, {\em UMT-domains and domains
with Pr\"{u}fer integral closure}, Comm. Algebra {\bf 26}, (1998),
1017--1039.


\bibitem{FHP} M. Fontana, J. Huckaba, and I. Papick, {\em Pr\"{u}fer domains}, New York, Marcel
Dekker, 1997.

\bibitem{FL} M. Fontana and K. A. Loper, {\em Nagata rings, Kronecker function
rings and related semistar operations}, Comm. Algebra {\bf 31}
(2003), 4775--4801.


\bibitem{G} R. Gilmer, {\em Multiplicative ideal theory}, New York, Dekker, 1972.

\bibitem{HZ} E. Houston and M. Zafrullah, {\em On $t$-invertibility, II}, Comm. Algebra
{\bf 17} (1989), 1955--1969.

\bibitem{K} I. Kaplansky, {\em Commutative rings}, rev. ed., Univ. Chicago Press, Chicago, 1974.

\bibitem{MM} S. Malik and J. L. Mott, {\em Strong S-domains}, J. Pure Appl. Algebra, {\bf
28} (1983), 249--264.

\bibitem{Mim} A. Mimouni, {\em Integral domains in which each ideal is a $w$-ideal}, Comm. Algebra
{\bf 33}, No. 5, (2005), 1345--1355.

\bibitem{S} P. Sahandi, {\em Semistar-Krull and valuative dimension of integral
domains}, Ricerche Mat., {\bf 58}, (2009), 219--242.

\bibitem{Sah} P. Sahandi, {\em Universally catenarian integral domains, strong S-domains and semistar operations},
Comm. Algebra {\bf 38}, No. 2, (2010), 673--683.

\bibitem{S1} P. Sahandi, {\em Semistar dimension of polynomial rings and Pr\"{u}fer-like domains},
Bull. Iranian Math. Soc., to appear, arXiv:0808.1331v2 [math.AC].

\bibitem{SAHA} P. Sahandi, {\em W-Jaffard domains in pullbacks}, J.
Algebra and its Applications, to appear,  arXiv:1003.1565v3
[math.AC].

\end{thebibliography}
\end{document}